\documentclass{amsart}
\usepackage{amssymb}
\usepackage{amsmath}
\usepackage{amsfonts}

\newtheorem{theorem}{Theorem}
\theoremstyle{plain}

\newtheorem{example}{Example}

\newtheorem{lemma}{Lemma}

\input{tcilatex}

\begin{document}
\title{ Inverse nodal problems for Dirac differential operators with jump
condition.}
\author{Baki Keskin}
\curraddr{Department of Mathematics, Faculty of Science, Cumhuriyet
University 58140 \\
Sivas, TURKEY}
\email{bkeskin@cumhuriyet.edu.tr}
\subjclass[2000]{34A55, 34L40, 34L05, 34K29, 34K10, }
\keywords{Dirac differential operator, discontinuity conditions inside the
interval, inverse nodal problem, uniqueness theorem,}

\begin{abstract}
This paper deals with an inverse nodal problem for the Dirac differential
operator with the discontinuity conditions inside the interval. We obtain a
new approach for asymptotic expressions of the solutions and prove that the
coefficients of \ the Dirac system can be determined uniquely by a dense
subset of the nodal points (zeros of the first component of the
eigenfunction). We also provide an algorithm for constructing the solution
of this inverse nodal problem.
\end{abstract}

\maketitle

\section{\textbf{Introduction}}

Consider the following boundary value problem $L$ generated by the system of
discontinuous Dirac differential equations%
\begin{equation}
BY^{\prime }(x)+\Omega (x)Y(x)=\mu Y(x),\text{ \ }x\in (0,\pi ),
\end{equation}%
with the boundary conditions%
\begin{eqnarray}
y_{1}(0)\sin \theta +y_{2}(0)\cos \theta &=&0\medskip \\
y_{1}(\pi )\sin \beta +y_{2}(\pi )\cos \beta &=&0
\end{eqnarray}%
and discontinuity conditions%
\begin{eqnarray}
y_{1}(\frac{\pi }{2}+0) &=&\sigma y_{1}(\frac{\pi }{2}-0) \\
y_{2}(\frac{\pi }{2}+0) &=&\sigma ^{-1}y_{2}(\frac{\pi }{2}-0).  \notag
\end{eqnarray}%
Here $0\leq <\beta $, $\theta <\pi ,$ $(\beta ,\theta \in 
\mathbb{R}
)$, $\sigma \in 
\mathbb{R}
^{+},$ $B=\left( 
\begin{array}{cc}
0 & 1 \\ 
-1 & 0%
\end{array}%
\right) ,$ $\Omega (x)=\left( 
\begin{array}{cc}
p(x) & 0 \\ 
0 & q(x)%
\end{array}%
\right) ,$ $Y(x)=(y_{1}(x),y_{2}(x))^{T},$ $\Omega (x)$ is real-valued
functions in $W_{2}^{1}(0,\pi )$ and $\mu $ is the spectral parameter$.$

The inverse spectral analysis play an important role in mathematics,
mathematical physics and have many applications in natural sciences.
Boundary value problems with discontinuity conditions inside the interval
often appear in mathematics, mechanics, physics, geophysics and other
branches of natural properties. The basic and comprehensive results about
the Dirac operators were given in \cite{Lev}. Classical inverse problems for
the Dirac operators have been extensively studied in various publications
(see \cite{Alb}, \cite{Gsy1}, \cite{Gus}, \cite{Hor}, \cite{Ksk2}, \cite{Ksk}
, \cite{wang2} and the references therein). Inverse nodal problem was first
proposed and solved for the Sturm--Liouville operator by McLaughlin \cite%
{mc1} in 1988. In this study, it has been shown that knowledge of a dense
subset of zeros of eigenfunctions ,called nodal points, uniquelly determine
the potential of the Sturm Liouville operator up to a constant. In 1989,
Hald and McLaughlin consider an inverse nodal Sturm--Liouville problem\ with
more general boundary conditions and give some numerical schemes for the
reconstruction of the potential from nodal points \cite{H}. Yang proposed an
algorithm to solve an inverse nodal problem for the Sturm--Liouville
operator in 1997 \cite{yang}. Such problems have been considered by several
researchers in (\cite{Br2}, \cite{But}, \cite{ch}, \cite{gul}, \cite{Hu}, 
\cite{guo1}, \cite{kesk3}, \cite{kesk4}, \cite{Law}, \cite{Ozkan}, \cite%
{wang}, \cite{wang3}, \cite{wei}, \cite{Yur}, \cite{yng1}, \cite{Yang3} and 
\cite{Yang4} ) and other works.The inverse nodal problems for the Dirac
operators with various boundary conditions have been studied and shown that
a dense subset of the zeros of the first component of the eigenfunctions
alone can determine the coefficients of discussed problem \cite{Guo}, \cite%
{keskin5}, \cite{Yang2} and \cite{Yang5}. Since, there are not sufficiently
good asymptotic expressions for the integral equations of the solutions of
the discontinuous Dirac operator, inverse nodal problems for this kind of
operator has not been considered before. The aim of this paper is to study
an inverse problem of recovering the coefficients of the discontinuous Dirac
system from nodal characteristics. In this paper, we have obtained a new
approach for calculating the asymptotic expressions of the solutions of the
considered problem. With the help of the these asymptotics, more accurate
expressions of the eigenvalues and zeros of the first component of the
eigenfunctions have been calculated. At the end of this paper we gave an
algorithm for reconstructing the operator by nodal data.

Let $\psi (x,\mu )=\left( 
\begin{array}{c}
\psi _{1}(x,\mu ) \\ 
\psi _{2}(x,\mu )%
\end{array}%
\right) $ be the solution of the system (1) under the initial condition $%
\psi (0,\mu )=\left( 
\begin{array}{c}
\cos \theta \\ 
-\sin \theta%
\end{array}%
\right) $. It is clear that for each fixed $x\in \left( 0,\pi \right) $, $%
\psi (x,\mu )$ is entire in $\mu $ and satisfies the following integral
equations:\newline
for $x<\dfrac{\pi }{2}$%
\begin{eqnarray}
\psi _{1}(x,\mu ) &=&\cos \left( \mu x-\theta \right) \\
&&-\int\limits_{0}^{x}\sin \mu (t-x)\psi
_{1}(t)p(t)dt+\int\limits_{0}^{x}\cos \mu (t-x)\psi _{2}(t)q(t)dt  \notag \\
\psi _{2}(x,\mu ) &=&\sin \left( \mu x-\theta \right) \\
&&-\int\limits_{0}^{x}\cos \mu (t-x)\psi
_{1}(t)p(t)dt-\int\limits_{0}^{x}\sin \mu (t-x)\psi _{2}(t)q(t)dt  \notag
\end{eqnarray}%
and for $x>\dfrac{\pi }{2}$%
\begin{eqnarray}
\psi _{1}(x,\mu ) &=&\sigma ^{+}\cos \left( \mu x-\theta \right) +\sigma
^{-}\cos \left( \mu (\pi -x)-\theta \right) \\
&&-\int\limits_{0}^{\pi /2}\left( \sigma ^{+}\sin \mu (t-x)+\sigma ^{-}\sin
\mu (x+t-\pi )\right) \psi _{1}(t)p(t)dt  \notag
\end{eqnarray}%
\begin{eqnarray*}
&&+\int\limits_{0}^{\pi /2}\sigma ^{+}\cos \mu (t-x)+\sigma ^{-}\cos \mu
(x+t-\pi )\psi _{2}(t)q(t)dt \\
&&-\int\limits_{\pi /2}^{x}\sin \mu (t-x)\psi _{1}(t)p(t)dt+\int\limits_{\pi
/2}^{x}\cos \mu (t-x)\psi _{2}(t)q(t)dt\medskip
\end{eqnarray*}%
\begin{eqnarray}
\psi _{2}(x,\mu ) &=&\sigma ^{+}\sin \left( \mu x-\theta \right) -\sigma
^{-}\sin \left( \mu (\pi -x)-\theta \right) \medskip  \notag \\
&&-\int\limits_{0}^{\pi /2}\left( \sigma ^{+}\cos \mu (t-x)-\sigma ^{-}\cos
\mu (x+t-\pi )\right) \psi _{1}(t)p(t)dt\medskip \\
&&-\int\limits_{0}^{\pi /2}\sigma ^{+}\sin \mu (t-x)-\sigma ^{-}\sin \mu
(x+t-\pi )\psi _{2}(t)q(t)dt\medskip  \notag \\
&&-\int\limits_{\pi /2}^{x}\cos \mu (t-x)\psi _{1}(t)p(t)dt-\int\limits_{\pi
/2}^{x}\sin \mu (t-x)\psi _{2}(t)q(t)dt  \notag
\end{eqnarray}%
where, $\sigma ^{\pm }=\dfrac{1}{2}\left( \sigma \pm \dfrac{1}{\sigma }%
\right) .$ In the case where $p(x)=V(x)+m,$ $q(x)=V(x)-m,$ $V(x)$ is a
potential function and $m$ is the mass of a particle, (1) is called a one
dimensional stationary Dirac system in relativistic Schr\"{o}dinger operator
in quantum theory. Throughout this paper, we put $\left( 
\begin{array}{c}
p(x) \\ 
q(x)%
\end{array}%
\right) =\left( 
\begin{array}{c}
V(x)+m \\ 
V(x)-m%
\end{array}%
\right) $ and without loss of generality we assume that $\int_{0}^{\pi
}V(t)dt=0.$

\section{Main Results}

\begin{theorem}
For $\left\vert \mu \right\vert \rightarrow \infty ,$ uniformly in $x,$the
functions $\psi _{1}(x,\mu )$ and $\psi _{2}(x,\mu )$ have the following
representations :

for $x<\dfrac{\pi }{2}$%
\begin{eqnarray}
\psi _{1}(x,\mu ) &=&\cos [\mu x-\rho (x)-\theta ]\medskip  \notag \\
&&+\dfrac{m\sin \theta }{\mu }\sin [\mu x-\rho (x)]\medskip \\
&&+\dfrac{m^{2}x}{2\mu }\sin [\mu x-\rho (x)-\theta ]+o(\dfrac{e^{\mid \tau
\mid x}}{\mu })  \notag
\end{eqnarray}%
\begin{eqnarray}
\psi _{2}(x,\mu ) &=&\sin [\mu x-\rho (x)-\theta ]\medskip  \notag \\
&&-\dfrac{m\cos \theta }{\mu }\sin [\mu x-\rho (x)]\medskip \\
&&-\dfrac{m^{2}x}{2\mu }\cos [\mu x-\rho (x)-\theta ]+o(\dfrac{e^{\mid \tau
\mid x}}{\mu })  \notag
\end{eqnarray}%
and for $x>\dfrac{\pi }{2}$%
\begin{eqnarray}
\psi _{1}(x,\mu ) &=&\sigma ^{+}\cos [\mu x-\rho (x)-\theta ]\medskip  \notag
\\
&&+\sigma ^{-}\cos [\mu x-\rho (x)-\mu \pi +2\rho \left( \frac{\pi }{2}%
\right) +\theta ]  \notag \\
&&+\dfrac{\sigma ^{+}m\sin \theta }{\mu }\sin [\mu x-\rho (x)] \\
&&-\dfrac{\sigma ^{-}m\sin \theta }{\mu }\sin [\mu x-\rho (x)-\mu \pi +2\rho
\left( \frac{\pi }{2}\right) ]  \notag \\
&&+\dfrac{\sigma ^{+}m^{2}x}{2\mu }\sin [\mu x-\rho (x)-\theta ]  \notag \\
&&-\dfrac{\sigma ^{-}m^{2}\left( \pi -x\right) }{2\mu }\sin [\mu x-\rho
(x)-\mu \pi +2\rho \left( \frac{\pi }{2}\right) +\theta ]+o(\dfrac{e^{\mid
\tau \mid x}}{\mu })  \notag
\end{eqnarray}%
\begin{eqnarray}
\psi _{2}(x,\mu ) &=&\sigma ^{+}\sin [\mu x-\rho (x)-\theta ]  \notag \\
&&+\sigma ^{-}\sin [\mu x-\rho (x)-\mu \pi +2\rho \left( \frac{\pi }{2}%
\right) +\theta ]  \notag \\
&&-\dfrac{\sigma ^{+}m\cos \theta }{\mu }\sin [\mu x-\rho (x)] \\
&&+\dfrac{\sigma ^{-}m\sin \theta }{\mu }\sin [\mu x-\rho (x)-\mu \pi +2\rho
\left( \frac{\pi }{2}\right) ]  \notag \\
&&-\dfrac{\sigma ^{+}m^{2}x}{2\mu }\cos [\mu x-\rho (x)-\theta ]  \notag \\
&&+\dfrac{\sigma ^{-}m^{2}\left( \pi -x\right) }{2\mu }\cos [\mu x-\rho
(x)-\mu \pi +2\rho \left( \frac{\pi }{2}\right) +\theta ]+o(\dfrac{e^{\mid
\tau \mid x}}{\mu })  \notag
\end{eqnarray}%
where, $\rho (x)=\dfrac{1}{2}\int\limits_{0}^{x}(p(t)+q(t))dt$ and $\tau =%
\func{Im}\mu .$\bigskip
\end{theorem}

\begin{proof}
Firstly, in order to apply successive approximation method to the equations
(5) and (6), put 
\begin{equation*}
\begin{array}{l}
\psi _{1,0}(x,\mu )=\cos \left( \mu x-\theta \right) \\ 
\psi _{2,0}(x,\mu )=\sin \left( \mu x-\theta \right) \\ 
\psi _{1,r+1}(x,\mu )=\int\limits_{0}^{x}\sin \mu (x-t)\psi
_{1,r}(t)p(t)dt+\int\limits_{0}^{x}\cos \mu (x-t)\psi _{2,r}(t)q(t)dt\medskip
\\ 
\psi _{2,r+1}(x,\mu )=-\int\limits_{0}^{x}\cos \mu (x-t)\psi
_{1,r}(t)p(t)dt+\medskip \int\limits_{0}^{x}\sin \mu (x-t)\psi
_{2,r}(t)q(t)dt\medskip%
\end{array}%
\end{equation*}%
then we have%
\begin{equation*}
\psi _{1,1}(x,\mu )=\sin \left( \mu x-\theta \right) \rho (x)+\dfrac{m\sin
\theta }{\mu }\sin \mu x+o(\dfrac{e^{\mid \tau \mid x}}{\mu })
\end{equation*}%
\begin{equation*}
\psi _{2,1}(x,\mu )=-\cos \left( \mu x-\theta \right) \rho (x)-\dfrac{m\cos
\theta }{\mu }\sin \mu x+o(\dfrac{e^{\mid \tau \mid x}}{\mu })
\end{equation*}%
and for $r\geq 1$%
\begin{eqnarray*}
\psi _{1,2r+1}(x,\mu ) &=&(-1)^{r}\sin \left( \mu x-\theta \right) \dfrac{%
\rho ^{2r+1}(x)}{(2r+1)!}\medskip \\
&&+\dfrac{(-1)^{r}m\sin \theta }{\mu }\sin \mu x\dfrac{\rho ^{2r}(x)}{(2r)!}%
\medskip \\
&&+\dfrac{(-1)^{r}m^{2}x\sin \theta }{2\mu }\sin \mu x\dfrac{\rho ^{2r-1}(x)%
}{(2r-1)!}\medskip \\
&&+\dfrac{(-1)^{r}m^{2}x\cos \theta }{2\mu }\cos \mu x\dfrac{\rho ^{2r-1}(x)%
}{(2r-1)!}+o(\dfrac{e^{\mid \tau \mid x}}{\mu })\medskip
\end{eqnarray*}%
\begin{eqnarray*}
\psi _{1,2r}(x,\mu ) &=&(-1)^{r}\cos \left( \mu x-\theta \right) \dfrac{\rho
^{2r}(x)}{(2r)!}\medskip \\
&&+\dfrac{(-1)^{r}m\sin \theta }{\mu }\cos \mu x\dfrac{\rho ^{2r-1}(x)}{%
(2r-1)!}\medskip \\
&&+\dfrac{(-1)^{r+1}m^{2}x\cos \theta }{2\mu }\sin \mu x\dfrac{\rho
^{2r-2}(x)}{(2r-2)!}\medskip \\
&&+\dfrac{(-1)^{r+1}m^{2}x\sin \theta }{2\mu }\cos \mu x\dfrac{\rho
^{2r-2}(x)}{(2r-2)!}+o(\dfrac{e^{\mid \tau \mid x}}{\mu })
\end{eqnarray*}%
\begin{eqnarray*}
\psi _{2,2r+1}(x,\mu ) &=&(-1)^{r+1}\cos \left( \mu x-\theta \right) \dfrac{%
\rho ^{2r+1}(x)}{(2r+1)!}\medskip \\
&&+\dfrac{(-1)^{r+1}m\cos \theta }{\mu }\sin \mu x\dfrac{\rho ^{2r}(x)}{(2r)!%
}\medskip \\
&&+\dfrac{(-1)^{r+1}m^{2}x\sin \theta }{2\mu }\cos \mu x\dfrac{\rho
^{2r-1}(x)}{(2r-1)!}\medskip \\
&&+\dfrac{(-1)^{r}m^{2}x\cos \theta }{2\mu }\sin \mu x\dfrac{\rho ^{2r-1}(x)%
}{(2r-1)!}+o(\dfrac{e^{\mid \tau \mid x}}{\mu })
\end{eqnarray*}%
\begin{eqnarray*}
\psi _{2,2r}(x,\mu ) &=&(-1)^{r}a_{2}\sin \left( \mu x-\theta \right) \dfrac{%
\rho ^{2r}(x)}{(2r)!}\medskip \\
&&+\dfrac{(-1)^{r+1}m\cos \theta }{\mu }\cos \mu x\dfrac{\rho ^{2r-1}(x)}{%
(2r-1)!}\medskip \\
&&+\dfrac{(-1)^{r}m^{2}x\sin \theta }{2\mu }\sin \mu x\dfrac{\rho ^{2r-2}(x)%
}{(2r-2)!}\medskip \\
&&+\dfrac{(-1)^{r}m^{2}x\cos \theta }{2\mu }\cos \mu x\dfrac{\rho ^{2r-2}(x)%
}{(2r-2)!}+o(\dfrac{e^{\mid \tau \mid x}}{\mu })
\end{eqnarray*}%
Thus, for $x<\dfrac{\pi }{2}$ the proof is clear from above estimates, for
sufficiently large $\left\vert \mu \right\vert $ and uniformly in $x.$

In order to make similar calculations for (7) and (8), put%
\begin{eqnarray*}
\psi _{1,0}(x,\mu ) &=&\sigma ^{+}\cos \left( \mu x-\theta \right) +\sigma
^{-}\cos \left( \mu (\pi -x)-\theta \right)  \\
\psi _{2,0}(x,\mu ) &=&\sigma ^{+}\sin \left( \mu x-\theta \right) -\sigma
^{-}\sin \left( \mu (\pi -x)-\theta \right)  \\
\psi _{1,r+1}(x,\mu ) &=&-\int\limits_{0}^{\pi /2}\left( \sigma ^{+}\sin \mu
(t-x)+\sigma ^{-}\sin \mu (x+t-\pi )\right) \psi _{1,r}(t)p(t)dt\medskip  \\
&&+\int\limits_{0}^{\pi /2}\left( \sigma ^{+}\cos \mu (t-x)+\sigma ^{-}\cos
\mu (x+t-\pi )\right) \psi _{2,r}(t)q(t)dt\medskip  \\
&&-\int\limits_{\pi /2}^{x}\sin \mu (t-x)\psi _{1}(t)p(t)dt+\int\limits_{\pi
/2}^{x}\cos \mu (t-x)\psi _{2,r}(t)q(t)dt
\end{eqnarray*}%
\begin{eqnarray*}
\psi _{2,r+1}(x,\mu ) &=&-\int\limits_{0}^{\pi /2}\left( \sigma ^{+}\cos \mu
(t-x)-\sigma ^{-}\cos \mu (x+t-\pi )\right) \psi _{1,r}(t)p(t)dt\medskip  \\
&&-\int\limits_{0}^{\pi /2}\left( \sigma ^{+}\sin \mu (t-x)-\sigma ^{-}\sin
\mu (x+t-\pi )\right) \psi _{2,r}(t)q(t)dt\medskip  \\
&&-\int\limits_{\pi /2}^{x}\cos \mu (t-x)\psi
_{1,r}(t)p(t)dt-\int\limits_{\pi /2}^{x}\sin \mu (t-x)\psi _{2,r}(t)q(t)dt,
\end{eqnarray*}%
then we obtain%
\begin{eqnarray*}
\psi _{1,1}(x,\mu ) &=&\sigma ^{+}\sin \left( \mu x-\theta \right) \rho
(x)-\sigma ^{-}\sin \left( \mu (\pi -x)-\theta \right) \left( \rho (\frac{%
\pi }{2})-\rho _{1}(x)\right) \medskip  \\
&&+\dfrac{\sigma ^{-}m\sin \theta }{\mu }\sin \mu (\pi -x)+\dfrac{\sigma
^{+}m\sin \theta }{\mu }\sin \mu x+o\left( \frac{e^{\left\vert \func{Im}\tau
\right\vert x}}{\mu }\right) \medskip 
\end{eqnarray*}%
\begin{eqnarray*}
\psi _{2,1}(x,\mu ) &=&-\sigma ^{+}\cos \left( \mu x-\theta \right) \rho
(x)+\sigma ^{-}\cos \left( \mu (\pi -x)-\theta \right) \left( \rho (\frac{%
\pi }{2})-\rho _{1}(x)\right) \medskip  \\
&&-\dfrac{\sigma ^{-}m\sin \theta }{\mu }\sin \mu (\pi -x)-\dfrac{\sigma
^{+}m\cos \theta }{\mu }\sin \mu x+o\left( \frac{e^{\left\vert \func{Im}\tau
\right\vert x}}{\mu }\right) 
\end{eqnarray*}%
and for $r\geq 1$%
\begin{eqnarray*}
\psi _{1,2r}(x,\mu ) &=&(-1)^{r}\sigma ^{+}\cos \left( \mu x-\theta \right) 
\dfrac{\rho ^{2r}(x)}{(2r)!}\medskip  \\
&&+(-1)^{r}\sigma \cos \left( \mu (\pi -x)-\theta \right) \frac{\left( \rho (%
\frac{\pi }{2})-\rho _{1}(x)\right) ^{2r}}{(2r)!}\medskip  \\
&&+(-1)^{r}\dfrac{\sigma ^{+}m\sin \theta }{\mu }\cos \mu x\frac{\rho
^{2r-1}(x)}{\left( 2r-1\right) !}\medskip  \\
&&+(-1)^{r}\dfrac{\sigma ^{-}m\sin \theta }{\mu }\cos \mu (\pi -x)\frac{%
\left( \rho (\frac{\pi }{2})-\rho _{1}(x)\right) ^{2r-1}}{\left( 2r-1\right)
!}\medskip  \\
&&-(-1)^{r}\dfrac{\sigma ^{+}m^{2}x}{2\mu }\sin \left( \mu x-\theta \right) 
\frac{\rho ^{2r-2}(x)}{\left( 2r-2\right) !}\medskip  \\
&&-(-1)^{r}\dfrac{\sigma ^{-}m^{2}}{2\mu }\left( \pi -x\right) \sin (\mu
(\pi -x)-\theta )\frac{\left( \rho (\frac{\pi }{2})-\rho _{1}(x)\right)
^{2r-2}}{\left( 2r-2\right) !}\medskip  \\
&&+o(\dfrac{e^{\mid \tau \mid x}}{\mu })
\end{eqnarray*}

\begin{eqnarray*}
\psi _{1,2r+1}(x,\mu ) &=&(-1)^{r}\sigma ^{+}\sin \left( \mu x-\theta
\right) \dfrac{\rho ^{2r+1}(x)}{(2r+1)!}\medskip \\
&&+(-1)^{r+1}\sigma \sin \left( \mu (\pi -x)-\theta \right) \frac{\left(
\rho (\frac{\pi }{2})-\rho _{1}(x)\right) ^{2r+1}}{(2r+1)!}\medskip \\
&&+(-1)^{r}\dfrac{\sigma ^{+}m\sin \theta }{\mu }\sin \mu x\frac{\rho
^{2r}(x)}{\left( 2r\right) !}\medskip \\
&&+(-1)^{r}\frac{\sigma ^{-}m\sin \theta }{\mu }\sin \mu (\pi -x)\frac{%
\left( \rho (\frac{\pi }{2})-\rho _{1}(x)\right) ^{2r}}{\left( 2r\right) !}
\\
&&+(-1)^{r}\dfrac{\sigma ^{+}m^{2}x}{2\mu }\cos \left( \mu x-\theta \right) 
\frac{\rho ^{2r-1}(x)}{\left( 2r-1\right) !}\medskip \\
&&+(-1)^{r}\dfrac{\sigma ^{-}m^{2}}{2\mu }\left( \pi -x\right) \cos (\mu
(\pi -x)-\theta )\frac{\left( \rho (\frac{\pi }{2})-\rho _{1}(x)\right)
^{2r-1}}{\left( 2r-1\right) !} \\
&&+o(\dfrac{e^{\mid \tau \mid x}}{\mu })
\end{eqnarray*}%
\begin{eqnarray*}
\psi _{2,2r}(x,\mu ) &=&(-1)^{r}\sigma ^{+}\sin \left( \mu x-\theta \right) 
\frac{\rho ^{2r}(x)}{\left( 2r\right) !}\medskip \\
&&+(-1)^{r+1}\sigma ^{-}\sin \left( \mu (\pi -x)-\theta \right) \frac{\left(
\rho (\frac{\pi }{2})-\rho _{1}(x)\right) ^{2r-1}}{\left( 2r-1\right) !}%
\medskip \\
&&+(-1)^{r+1}\dfrac{\sigma ^{+}m\cos \theta }{\mu }\cos \mu x\frac{\rho
^{2r-1}(x)}{\left( 2r-1\right) !}\medskip \\
&&+(-1)^{r+1}\frac{\sigma ^{-}m\sin \theta }{\mu }\cos \mu (\pi -x)\frac{%
\left( \rho (\frac{\pi }{2})-\rho _{1}(x)\right) ^{2r-1}}{(2r-1)!}\medskip \\
&&+(-1)^{r}\dfrac{\sigma ^{+}m^{2}x}{2\mu }\cos \left( \mu x-\theta \right) 
\frac{\rho ^{2r-2}(x)}{\left( 2r-2\right) !}\medskip \\
&&+(-1)^{r}\dfrac{\sigma ^{-}m^{2}}{2\mu }\left( \pi -x\right) \cos (\mu
(\pi -x)-\theta )\frac{\left( \rho (\frac{\pi }{2})-\rho _{1}(x)\right)
^{2r-2}}{\left( 2r-2\right) !}\medskip \\
&&+o\left( \frac{e^{\left\vert \func{Im}\tau \right\vert x}}{\mu }\right)
\end{eqnarray*}%
\begin{eqnarray*}
\psi _{2,2r+1}(x,\mu ) &=&(-1)^{r+1}\sigma ^{+}\cos \left( \mu x-\theta
\right) \dfrac{\rho ^{2r+1}(x)}{(2r+1)!}\medskip \\
&&+(-1)^{r}\cos \left( \mu (\pi -x)-\theta \right) \frac{\left( \rho (\frac{%
\pi }{2})-\rho _{1}(x)\right) ^{2r+1}}{(2r+1)!}\medskip \\
&&+(-1)^{r+1}\dfrac{\sigma ^{+}m\cos \theta }{2\mu }\sin \mu x\frac{\rho
^{2r}(x)}{\left( 2r\right) !}\medskip \\
&&+(-1)^{r+1}\frac{\sigma ^{-}m\sin \theta }{\mu }\cos \mu (\pi -x)\frac{%
\left( \rho (\frac{\pi }{2})-\rho _{1}(x)\right) ^{2r}}{\left( 2r\right) !}
\\
&&+(-1)^{r}\dfrac{\sigma ^{+}m^{2}x}{2\mu }\sin \left( \mu x-\theta \right) 
\frac{\rho ^{2r-1}(x)}{\left( 2r-1\right) !} \\
&&+(-1)^{r+1}\dfrac{\sigma ^{-}m^{2}}{\mu }\left( \pi -x\right) \sin (\mu
(\pi -x)-\theta )\frac{\left( \rho (\frac{\pi }{2})-\rho _{1}(x)\right)
^{2r-1}}{\left( 2r-1\right) !} \\
&&+o\left( \frac{e^{\left\vert \func{Im}\tau \right\vert x}}{\mu }\right)
\end{eqnarray*}%
where $\rho _{1}(x)=\rho (x)-\rho (\dfrac{\pi }{2})$. This gives the proof
of the second part of theorem.
\end{proof}

Define the entire function $\Delta (\mu )$ by 
\begin{equation}
\Delta (\mu )=\sin \beta \psi _{1}(\pi ,\mu )+\cos \beta \psi _{2}(\pi ,\mu
).
\end{equation}%
The function $\Delta (\mu )$ is entire in $\lambda $ and called the
characteristic function of the problem $L$. Its zeros $\left\{ \mu
_{n}\right\} _{n\in 
\mathbb{Z}
}$ coincide with the eigenvalues of the problem $L$. It follows from (11)
and (12) that, for $\left\vert \mu \right\vert \rightarrow \infty $

\begin{eqnarray}
\Delta (\mu ) &=&\sigma ^{+}\sin (\mu \pi -\theta +\beta )+\sigma ^{-}\sin
(2\rho (\frac{\pi }{2})+\theta +\beta )\medskip  \notag \\
&&-\frac{\sigma ^{+}m}{\mu }\cos (\theta +\beta )\sin \mu \pi +\dfrac{\sigma
^{-}m}{\mu }(\cos \beta -\cos \theta )\sin \theta \sin (2\rho (\frac{\pi }{2}%
))\medskip \\
&&-\dfrac{\sigma ^{+}m^{2}\pi }{2\mu }\cos (\mu \pi -\theta +\beta )+o(%
\dfrac{e^{\mid \tau \mid x}}{\mu })  \notag
\end{eqnarray}%
$\medskip $ Let us introduce the auxiliary function

$\ \Delta _{0}(\mu )=\sigma ^{+}\sin (\mu \pi -\theta +\beta )+\sigma
^{-}\sin (2\rho (\frac{\pi }{2})+\theta +\beta )$\newline
and denote the set of zeros of it by $\left\{ \mu _{n}^{0}\right\} _{n\in 
\mathbb{Z}
}$. Write,

$\Delta _{0}(\mu _{n}^{0})=\sigma ^{+}\sin (\mu _{n}^{0}\pi -\theta +\beta
)+\sigma ^{-}\sin (2\rho (\frac{\pi }{2})+\theta +\beta )=0,$\newline
then,

$\medskip \sigma ^{+}\sin (\mu _{n}^{0}\pi -\theta +\beta )=-\sigma ^{-}\sin
(2\rho (\frac{\pi }{2})+\theta +\beta ),$\newline
this implies,

$\medskip \mu _{n}^{0}=n-\dfrac{\beta -\theta }{\pi }+\dfrac{\left(
-1\right) ^{n}}{\pi }\arcsin \gamma ,$\newline
where, $\gamma =[-\dfrac{\sigma ^{-}}{\sigma ^{+}}\sin (2\rho (\frac{\pi }{2}%
)+\theta +\beta )].$

Using well known method (see, for example \cite{G}), since $\mu _{n}=\mu
_{n}^{0}+o(1),$ we obtain%
\begin{equation*}
\mu _{n}=n[1-\dfrac{\beta -\theta }{n\pi }+\dfrac{\left( -1\right) ^{n}}{%
n\pi }\arcsin \gamma +o(\dfrac{1}{n})]
\end{equation*}%
for $\left\vert n\right\vert \rightarrow \infty .$

\begin{lemma}
The function $\psi _{1}(x,\mu _{n})$ has exactly $n$ nodes $\left\{
x_{n}^{j}:n\geq 1,\text{ }j=\overline{0,n-1}\right\} $ i.e.,\newline
$0<x_{n}^{0}<x_{n}^{1}<...<x_{n}^{n-1}<\pi $. Moreover, the nodes $\left\{
x_{n}^{j}\right\} $ has the following representations as $n\rightarrow
\infty $ uniformly in $j$
\end{lemma}

for $x<\dfrac{\pi }{2}$%
\begin{eqnarray*}
x_{n}^{j} &=&\dfrac{j\pi }{n}+\dfrac{\beta -\theta -(-1)^{n}\arcsin \gamma }{%
n\pi }\dfrac{j\pi }{n}+\dfrac{\rho (x_{n}^{j})}{n}\medskip \\
&&-\dfrac{\cot \theta }{n}+\dfrac{\left\{ \beta -\theta -(-1)^{n}\arcsin
\gamma \right\} \rho (x_{n}^{j})}{n^{2}\pi }\medskip \\
&&+\frac{m\cot \theta }{n^{2}}+\dfrac{\left\{ \beta -\theta -(-1)^{n}\arcsin
\gamma \right\} \cot \theta }{n^{2}\pi }\medskip \\
&&+\frac{m^{2}x\csc ^{2}\theta }{2n^{2}}+o(\dfrac{1}{n^{2}})
\end{eqnarray*}%
and for $x>\dfrac{\pi }{2}$%
\begin{eqnarray*}
x_{n}^{j} &=&\dfrac{j\pi }{n}+\dfrac{\beta -\theta -(-1)^{n}\arcsin \gamma }{%
n\pi }\dfrac{j\pi }{n}+\dfrac{\rho (x_{n}^{j})}{n}\medskip \\
&&\medskip +\dfrac{-\sigma ^{+}\cos \theta -\sigma ^{-}(-1)^{n}\cos (\beta
-\left( -1\right) ^{n}\arcsin \gamma +2\rho (\frac{\pi }{2}))}{nT_{1}^{\ast }%
} \\
&&\medskip +\dfrac{\left\{ \beta -\theta -(-1)^{n}\arcsin \gamma \right\}
\rho (x_{n}^{j})}{n^{2}\pi } \\
&&+\dfrac{\left\{ \beta -\theta -(-1)^{n}\arcsin \gamma \right\} \left\{
-\sigma ^{+}\cos \theta -\sigma ^{-}(-1)^{n}\cos (\beta -\left( -1\right)
^{n}\arcsin \gamma +2\rho (\frac{\pi }{2}))\right\} }{n^{2}\pi T_{1}^{\ast }}%
\medskip \\
&&+\dfrac{\sigma ^{+}\cos \theta +\sigma ^{-}(-1)^{n}\cos (\beta -\left(
-1\right) ^{n}\arcsin \gamma +2\rho (\frac{\pi }{2}))}{n^{2}T_{1}^{\ast }}%
\dfrac{T_{2}^{\ast }}{T_{1}^{\ast }}+\dfrac{M^{\ast }}{n^{2}T_{1}^{\ast }}+o(%
\dfrac{1}{n^{2}})
\end{eqnarray*}%
for $\left\vert n\right\vert \rightarrow \infty $, where,

$T_{1}^{\ast }=\sigma ^{+}\sin \theta +\sigma ^{-}(-1)^{n}\sin (\beta
-\left( -1\right) ^{n}\arcsin \gamma +2\rho (\frac{\pi }{2}))\medskip $

$T_{2}^{\ast }=\sigma ^{+}m\sin \theta +\sigma ^{-}m\sin \theta (-1)^{n}\cos
(\beta -\left( -1\right) ^{n}\arcsin \gamma +2\rho (\frac{\pi }{2}))\medskip 
$

$+\dfrac{1}{2}\sigma ^{+}m^{2}x_{n}^{j}\cos \theta +\dfrac{1}{2}\sigma
^{-}m^{2}(\pi -x_{n}^{j})(-1)^{n}\cos (\beta -\left( -1\right) ^{n}\arcsin
\gamma +2\rho (\frac{\pi }{2}))\medskip $

$M^{\ast }=(-1)^{n}\sigma ^{-}m\sin \theta \sin (\beta -\left( -1\right)
^{n}\arcsin \gamma +2\rho (\frac{\pi }{2}))+\dfrac{\sigma ^{+}m^{2}}{2}%
x_{n}^{j}\sin \theta +\dfrac{(-1)^{n}\sigma ^{-}m^{2}}{2}(\pi
-x_{n}^{j})\sin (\beta -\left( -1\right) ^{n}\arcsin \gamma +2\rho (\frac{%
\pi }{2}))\medskip $

\begin{proof}
Let us prove the asymptotic expansion for $x>\dfrac{\pi }{2}.$ Using the
formula (11), the following asymptotic relation can be written for
sufficiently large $\left\vert n\right\vert $%
\begin{eqnarray}
\psi _{1}(x_{n}^{j},\mu _{n}) &=&\sigma ^{+}\cos [\mu _{n}x_{n}^{j}-\rho
(x_{n}^{j})-\theta ]  \notag \\
&&+\sigma ^{-}\cos [\mu _{n}x_{n}^{j}-\rho (x_{n}^{j})-\mu _{n}\pi +2\rho
\left( \frac{\pi }{2}\right) +\theta ]  \notag \\
&&+\dfrac{\sigma ^{+}m\sin \theta }{\mu _{n}}\sin [\mu _{n}x_{n}^{j}-\rho
(x_{n}^{j})] \\
&&-\dfrac{\sigma ^{-}m\sin \theta }{\mu _{n}}\sin [\mu _{n}x-\rho
(x_{n}^{j})-\mu _{n}\pi +2\rho \left( \frac{\pi }{2}\right) ]  \notag \\
&&+\dfrac{\sigma ^{+}m^{2}x_{n}^{j}}{2\mu _{n}}\sin [\mu _{n}x_{n}^{j}-\rho
(x_{n}^{j})-\theta ]  \notag \\
&&-\dfrac{\sigma ^{-}m^{2}\left( \pi -x_{n}^{j}\right) }{2\mu _{n}}\sin [\mu
_{n}x_{n}^{j}-\rho (x_{n}^{j})-\mu _{n}\pi +2\rho \left( \frac{\pi }{2}%
\right) +\theta ]  \notag \\
&&+o(\dfrac{e^{\mid \tau \mid x_{n}^{j}}}{\mu _{n}})  \notag
\end{eqnarray}%
\medskip If we write $\psi _{1}(x_{n}^{j},\mu _{n})=0,$ then we get$\medskip 
$\newline
$\tan (\mu _{n}x_{n}^{j}-\rho (x_{n}^{j}))[\sigma ^{+}\sin \theta +\sigma
^{-}\sin (\mu _{n}\pi -2\rho \left( \frac{\pi }{2}\right) -\theta )\medskip $

$+\dfrac{\sigma ^{+}m}{\mu _{n}}\sin \theta +\dfrac{\sigma ^{-}m}{\mu _{n}}%
\sin \theta \cos (\mu _{n}\pi -2\rho \left( \frac{\pi }{2}\right) \medskip $

$+\dfrac{\sigma ^{+}m^{2}}{2\mu _{n}}x_{n}^{j}\cos \theta +\dfrac{\sigma
^{-}m^{2}}{2\mu _{n}}(\pi -x_{n}^{j})\cos (\mu _{n}\pi -2\rho \left( \frac{%
\pi }{2}\right) -\theta )]$\medskip

$=-\sigma ^{+}\cos \theta -\sigma ^{-}\cos (\mu _{n}\pi -2\rho \left( \frac{%
\pi }{2}\right) -\theta )+\dfrac{\sigma ^{-}m}{\mu _{n}}\sin \theta \sin
(\mu _{n}\pi -2\rho \left( \frac{\pi }{2}\right) )\medskip $

$+\dfrac{\sigma ^{+}m^{2}}{2\mu _{n}}x_{n}^{j}\sin \theta +\dfrac{\sigma
^{-}m^{2}}{2\mu _{n}}(\pi -x_{n}^{j})\sin (\mu _{n}\pi -2\rho \left( \frac{%
\pi }{2}\right) -\theta )+o(\dfrac{e^{\mid \tau \mid x_{n}^{j}}}{\mu _{n}}%
)\medskip $\newline
which is equivalent to$\medskip $

$\tan (\mu _{n}x_{n}^{j}-\rho (x_{n}^{j}))=[\sigma ^{+}\sin \theta +\sigma
^{-}\sin (\mu _{n}\pi -2\rho \left( \frac{\pi }{2}\right) -\theta )\medskip $

$+\dfrac{\sigma ^{+}m}{\mu _{n}}\sin \theta +\dfrac{\sigma ^{-}m}{\mu _{n}}%
\sin \theta \cos \left( \mu _{n}\pi -2\rho \left( \frac{\pi }{2}\right)
\right) \medskip $

$+\dfrac{\sigma ^{+}m^{2}}{2\mu _{n}}x_{n}^{j}\cos \theta +\dfrac{\sigma
^{-}m^{2}}{2\mu _{n}}(\pi -x_{n}^{j})\cos (\mu _{n}\pi -2\rho \left( \frac{%
\pi }{2}\right) -\theta )]^{-1}\medskip $

$\times \lbrack -\sigma ^{+}\cos \theta -\sigma ^{-}\cos (\mu _{n}\pi -2\rho
\left( \frac{\pi }{2}\right) -\theta )+\dfrac{\sigma ^{-}m}{\mu _{n}}\sin
\theta \sin (\mu _{n}\pi -2\rho \left( \frac{\pi }{2}\right) )\medskip $

$+\dfrac{\sigma ^{+}m^{2}}{2\mu _{n}}x_{n}^{j}\sin \theta +\dfrac{\sigma
^{-}m^{2}}{2\mu _{n}}(\pi -x_{n}^{j})\sin (\mu _{n}\pi -2\rho \left( \frac{%
\pi }{2}\right) -\theta )+o(\dfrac{e^{\mid \tau \mid x_{n}^{j}}}{\mu _{n}})]$%
\bigskip

taking into consideration Taylor expansion for the function arctangent, we
have$\medskip $\newline
$\mu _{n}x_{n}^{j}-\rho (x_{n}^{j})=j\pi +\dfrac{-\sigma ^{+}\cos \theta
-\sigma ^{-}\cos (\mu _{n}\pi -2\rho \left( \frac{\pi }{2}\right) -\theta )}{%
T_{1}}\medskip $

$+\dfrac{\sigma ^{+}\cos \theta +\sigma ^{-}\cos (\mu _{n}\pi -2\rho \left( 
\frac{\pi }{2}\right) -\theta )}{T_{1}}\dfrac{T_{2}}{\mu _{n}T_{1}}+\dfrac{M%
}{T_{1}\mu _{n}}+o(\dfrac{1}{\mu _{n}})\medskip $\newline
where$\medskip $

$T_{1}=\sigma ^{+}\sin \theta +\sigma ^{-}\sin (\mu _{n}\pi -2\rho \left( 
\frac{\pi }{2}\right) -\theta )\medskip $

$T_{2}=\sigma ^{+}m\sin \theta +\sigma ^{-}m\sin \theta \cos \left( \mu
_{n}\pi -2\rho \left( \frac{\pi }{2}\right) \right) \medskip $

$+\dfrac{1}{2}\sigma ^{+}m^{2}x_{n}^{j}\cos \theta +\dfrac{1}{2}\sigma
^{-}m^{2}(\pi -x_{n}^{j})\cos (\mu _{n}\pi -2\rho \left( \frac{\pi }{2}%
\right) -\theta )\medskip $\newline
and$\medskip $\newline
$M=\sigma ^{-}m\sin \theta \sin (\mu _{n}\pi -2\rho \left( \frac{\pi }{2}%
\right) +\dfrac{\sigma ^{+}m^{2}}{2}x_{n}^{j}\sin \theta +\dfrac{\sigma
^{-}m^{2}}{2}(\pi -x_{n}^{j})\sin (\mu _{n}\pi -2\rho \left( \frac{\pi }{2}%
\right) -\theta ),\medskip $\newline
this implies$\medskip $\newline
$x_{n}^{j}=\dfrac{j\pi }{\mu _{n}}+\dfrac{\rho (x_{n}^{j})}{\mu _{n}}+\dfrac{%
-\sigma ^{+}\cos \theta -\sigma ^{-}\cos (\mu _{n}\pi -2\rho \left( \frac{%
\pi }{2}\right) -\theta )}{\mu _{n}T_{1}}\medskip $

$+\dfrac{\sigma ^{+}\cos \theta +\sigma ^{-}\cos (\mu _{n}\pi -2\rho \left( 
\frac{\pi }{2}\right) -\theta )}{\mu _{n}T_{1}}\dfrac{T_{2}}{\mu _{n}T_{1}}+%
\dfrac{M}{T_{1}\mu _{n}^{2}}+o(\dfrac{1}{\mu _{n}^{2}})\medskip $

furthermore, if we take into account the asypmtotic formula $\medskip $

$\mu _{n}^{-1}=\dfrac{1}{n}\left\{ 1+\dfrac{\beta -\theta }{n\pi }-\dfrac{%
\left( -1\right) ^{n}}{n\pi }\arcsin \gamma +o(\dfrac{1}{n})\right\}
\medskip $ and

then, we conclude that the following equality holds$\medskip $

$x_{n}^{j}=\dfrac{j\pi }{n}+\dfrac{\left( \beta -\theta \right) -\left(
-1\right) ^{n}\arcsin \gamma }{n\pi }\dfrac{j\pi }{n}+\dfrac{\rho (x_{n}^{j})%
}{n}\medskip $

$+\dfrac{\left\{ \left( \beta -\theta \right) -\left( -1\right) ^{n}\arcsin
\gamma \right\} \rho (x_{n}^{j})}{n^{2}\pi }\medskip $

$+\dfrac{-\sigma ^{+}\cos \theta -\sigma ^{-}(-1)^{n}\cos (\beta -\left(
-1\right) ^{n}\arcsin \gamma +2\rho (\frac{\pi }{2}))}{nT_{1}^{\ast }}%
\medskip $

$+\dfrac{\left\{ \left( \beta -\theta \right) -(-1)^{n}\arcsin \gamma
\right\} \left\{ -\sigma ^{+}\cos \theta -\sigma ^{-}(-1)^{n}\cos (\beta
-\left( -1\right) ^{n}\arcsin \gamma +2\rho (\frac{\pi }{2}))\right\} }{%
n^{2}\pi T_{1}^{\ast }}\medskip $

$+\dfrac{\sigma ^{+}\cos \theta +\sigma ^{-}(-1)^{n}\cos (\beta -\left(
-1\right) ^{n}\pi \arcsin \gamma +2\rho (\frac{\pi }{2}))}{n^{2}T_{1}^{\ast }%
}\dfrac{T_{2}^{\ast }}{T_{1}^{\ast }}+\dfrac{M^{\ast }}{n^{2}T_{1}^{\ast }}%
+o(\dfrac{1}{n^{2}})\medskip $\newline
\end{proof}

\begin{theorem}
Let's denote the set of zeros of the first components of the eigenfunctions
in $\left( 0,\pi \right) $ by $E$ and $E_{0}=\left\{ x_{n}^{j}:\text{ }n=2k,%
\text{ }k\in 
\mathbb{N}
\right\} $ be the dense subset of $E.$ For each fixed $x\in \left( 0,\pi
\right) ,$ one can choose a sequence $\left( x_{n}^{j(n)}\right) \subset
E_{0}$ so that $x_{n}^{j(n)}\rightarrow $ $x.$ Then the following limits
exist and finite and the corresponding equalities hold:%
\begin{equation}
\underset{\left\vert n\right\vert \rightarrow \infty }{\lim }n\left(
x_{n}^{j(n)}-\dfrac{j\pi }{n}\right) =\left\{ 
\begin{array}{c}
\Phi _{1}(x),\text{ }x<\frac{\pi }{2} \\ 
\Phi _{2}(x),\text{ }x>\frac{\pi }{2},%
\end{array}%
\right.
\end{equation}%
\begin{eqnarray}
&&\left. 
\begin{array}{c}
\underset{\left\vert n\right\vert \rightarrow \infty }{\lim }n^{2}\left(
x_{n}^{j(n)}-\dfrac{j\pi }{n}-\dfrac{\beta -\theta -\arcsin \gamma }{n\pi }%
\dfrac{j\pi }{n}-\dfrac{\rho (x_{n}^{j})}{n}-\dfrac{\cot \theta }{n}\right)
\\ 
=\Psi _{1}(x)\text{, }x<\frac{\pi }{2},%
\end{array}%
\right. \medskip  \notag \\
&&\left. \underset{\left\vert n\right\vert \rightarrow \infty }{\lim }%
n^{2}\left( x_{n}^{j(n)}-\dfrac{j\pi }{n}-\dfrac{\beta -\theta -\arcsin
\gamma }{n\pi }\dfrac{j\pi }{n}-\dfrac{\rho (x_{n}^{j})}{n}\right. \right.
\medskip \\
&&\left. 
\begin{array}{c}
\left. +\dfrac{\sigma ^{+}\cos \theta +\sigma ^{-}\cos (\beta -\arcsin
\gamma +2\rho (\frac{\pi }{2}))}{nT_{1}^{\ast \ast }}\right\} \\ 
=\Psi _{2}(x)\text{, }x>\frac{\pi }{2}%
\end{array}%
\right. ,  \notag
\end{eqnarray}%
then%
\begin{eqnarray}
\Phi _{1}(x) &=&\rho (x)+\dfrac{\beta -\theta -\arcsin \gamma }{\pi }x-\cot
\theta  \notag \\
\Phi _{2}(x) &=&\rho (x)+\dfrac{\beta -\theta -\arcsin \gamma }{\pi }x 
\notag \\
&&+\dfrac{-\sigma ^{+}\cos \theta -\sigma ^{-}\cos (\beta -\arcsin \gamma
+2\rho (\frac{\pi }{2}))}{T_{1}^{\ast \ast }} \\
\Psi _{1}(x) &=&\dfrac{\left\{ \beta -\theta -\arcsin \gamma \right\} \rho
(x)}{\pi }+\dfrac{\left\{ \beta -\theta -\arcsin \gamma \right\} \cot \theta 
}{\pi }  \notag \\
&&+m\cot \theta +m^{2}x\csc ^{2}\theta  \notag \\
\Psi _{2}(x) &=&\dfrac{\left\{ \beta -\theta -\arcsin \gamma \right\}
\left\{ -\sigma ^{+}\cos \theta -\sigma ^{-}\cos (\beta -\arcsin \gamma
+2\rho (\frac{\pi }{2}))\right\} \medskip }{\pi T_{1}^{\ast \ast }}  \notag
\\
&&+\dfrac{\sigma ^{+}\cos \theta +\sigma ^{-}\cos (\beta -\arcsin \gamma
+2\rho (\frac{\pi }{2}))}{T_{1}^{\ast \ast }}\dfrac{T_{2}^{\ast \ast }}{%
T_{1}^{\ast \ast }}+\dfrac{M^{\ast \ast }}{T_{1}^{\ast \ast }},  \notag
\end{eqnarray}%
where $T_{1}^{\ast \ast }=\sigma ^{+}\sin \theta +\sigma ^{-}\sin (\beta
-\arcsin \gamma +2\rho (\frac{\pi }{2}))\medskip $\newline
$T_{2}^{\ast \ast }=\sigma ^{+}m\sin \theta +\sigma ^{-}m\sin \theta \cos
(\beta -\arcsin \gamma +2\rho (\frac{\pi }{2}))$ and\newline
$M^{\ast \ast }=\sigma ^{-}m\sin \theta \sin (\beta -\arcsin \gamma +2\rho (%
\frac{\pi }{2}))+\dfrac{\sigma ^{+}m^{2}}{2}x_{n}^{j}\sin \theta +\dfrac{%
\sigma ^{-}m^{2}}{2}(\pi -x_{n}^{j})\sin (\beta -\arcsin \gamma +2\rho (%
\frac{\pi }{2})),$\newline
Let us put $\Phi (x)=\left\{ 
\begin{array}{l}
\Phi _{1}(x),\text{ }x<\frac{\pi }{2} \\ 
\Phi _{2}(x),\text{ }x>\frac{\pi }{2}%
\end{array}%
\right. $ and $\Psi (x)=\left\{ 
\begin{array}{l}
\Psi _{1}(x),\text{ }x<\frac{\pi }{2} \\ 
\Psi _{2}(x),\text{ }x>\frac{\pi }{2}%
\end{array}%
\right. $. Now, we can give the theorem which contains the algorithm of
{}how the potential is reconstructed by the given subset of nodal points.
\end{theorem}

\begin{theorem}
The given dense nodal subset $E_{0}$ uniquely determines the potential
function $\Omega (x)$ a.e. on $\left( 0,\pi \right) $ and the coefficient $%
\theta $ of the boundary conditions of the problem $L$. Moreover, $V(x),$ $%
m, $ and $\theta $ can be reconstructed via the algorithm:

$\left( \mathbf{1}\right) $ fix $x\in (0,\pi ),$ choose a sequence $\left(
x_{n}^{j(n)}\right) \subset E_{0}$ such that $x_{n}^{j(n)}\rightarrow x;$

$\left( \mathbf{2}\right) $\textbf{\ \ }find the function $\Phi (x)$ from
(16) and calculate 
\begin{eqnarray*}
&&\left. \theta =\func{arccot}(-\Phi _{1}(0))\right. \medskip \\
&&\left. V(x)=\medskip \Phi ^{\prime }(x)-\dfrac{\Phi _{2}(\pi )+\Phi _{1}(%
\frac{\pi }{2})-\Phi _{2}(\frac{\pi }{2})-\Phi _{1}(0)}{\pi }\right. \\
&&\left. {}\right.
\end{eqnarray*}%
$\left( \mathbf{3}\right) $\textbf{\ }find the function $\Psi (x)$ from (17)
and calculate%
\begin{equation*}
\medskip m=\frac{\Phi _{1}(0)(\Phi _{2}(\pi )+\Phi _{1}(\frac{\pi }{2})-\Phi
_{2}(\frac{\pi }{2})-\Phi _{1}(0))+\pi \Psi _{1}(0)}{-\pi \Phi _{1}(0)}.
\end{equation*}
\end{theorem}

\begin{example}
In the interval $(0,\pi ),$ let $\left\{ x_{n}^{j(n)}\right\} \subset E_{0}$
be the dense subset of even indexed nodal points given by the following
formulae\newline
\begin{eqnarray*}
&&\left. x_{n}^{j(n)}=\frac{j(n)\pi }{n}+\frac{-\sin \frac{j(n)\pi }{n}-\cot
1}{n}+\frac{2\cot 1+2\dfrac{j(n)\pi }{n}\csc ^{2}1}{n^{2}}+o\left( \frac{1}{%
n^{2}}\right) ,\text{ }x<\frac{\pi }{2}\right. \\
&&\left. x_{n}^{j(n)}=\frac{j(n)\pi }{n}+\frac{-\sin \frac{j(n)\pi }{n}%
-4\cot 1}{n}+\right. \\
&&\left. \frac{+20\cot 1+12\cos 1\cot 1+12\pi \cot ^{2}1-3\sin 1-3\pi +8%
\dfrac{j(n)\pi }{n}\csc ^{2}1}{n^{2}}+o\left( \frac{1}{n^{2}}\right) ,\text{ 
}x>\frac{\pi }{2}\right.
\end{eqnarray*}%
from (18) , one can calculate that,%
\begin{eqnarray*}
\Phi _{1}\left( x\right) &=&\underset{n\rightarrow \infty }{\lim }\left(
x_{n}^{j(n)}-\frac{j(n)\pi }{n}\right) n=-\sin x-\cot 1\medskip \\
\Phi _{2}\left( x\right) &=&\underset{n\rightarrow \infty }{\lim }\left(
x_{n}^{j(n)}-\frac{j(n)\pi }{n}\right) n=-\sin x-4\cot 1\medskip
\end{eqnarray*}%
\begin{eqnarray*}
\Psi _{1}\left( x\right) &=&\underset{n\rightarrow \infty }{\lim }\left\{
x_{n}^{j(n)}-\frac{j(n)\pi }{n}-\frac{\frac{1}{4}\left( \frac{j(n)\pi }{n}%
\right) ^{2}-\frac{\pi }{4}\dfrac{j(n)\pi }{n}-1}{n}\right\} n^{2} \\
&=&2\cot 1+2x\csc ^{2}1\medskip \\
\Psi _{2}\left( x\right) &=&\underset{n\rightarrow \infty }{\lim }\left\{
x_{n}^{j(n)}-\frac{j(n)\pi }{n}-\frac{\frac{\left( \frac{j(n)\pi }{2n}%
\right) ^{2}}{4}-\frac{\pi }{4}\dfrac{j(n)\pi }{n}-1}{2n}\right\} n^{2} \\
&=&20\cot 1+12\cos 1\cot 1+12\pi \cot ^{2}1-3\sin 1-3\pi +8x\csc ^{2}1 \\
&&\medskip
\end{eqnarray*}%
Therefore, it is obtained \ by using the algorithm in Therem 2,%
\begin{equation*}
\theta =\func{arccot}(-\Phi _{1}(0))=1\medskip
\end{equation*}%
\begin{equation*}
V(x)=\medskip \Phi ^{\prime }(x)-\dfrac{\Phi _{2}(\pi )+\Phi _{1}(\frac{\pi 
}{2})-\Phi _{2}(\frac{\pi }{2})-\Phi _{1}(0)}{\pi }=-\cos x\medskip
\end{equation*}%
\begin{equation*}
\medskip m=\frac{\Phi _{1}(0)(\Phi _{2}(\pi )+\Phi _{1}(\frac{\pi }{2})-\Phi
_{2}(\frac{\pi }{2})-\Phi _{1}(0))+\pi \Psi _{1}(0)}{-\pi \Phi _{1}(0)}=2.
\end{equation*}
\end{example}

\textbf{DATA AVAILABILITY}

Data sharing is not applicable to this article as no new data were created
or analyzed in this study.

\end{document}